\documentclass[11pt]{amsart}
\newcommand{\preprint}[1]{}

\newcommand{\hide}[1]{}

\usepackage{amssymb}
\usepackage{amsbsy}
\usepackage{amscd}
\usepackage{amsmath}
\usepackage{amsthm}
\usepackage{geometry}

\input xy
\xyoption{all}

\numberwithin{equation}{section}

\theoremstyle{plain}
\newtheorem{thm}{Theorem}[section]
\newtheorem{prop}[thm]{Proposition}
\newtheorem{claim}[thm]{Claim}

\newtheorem{conj}[thm]{Conjecture}
\newtheorem{cor}[thm]{Corollary}
\newtheorem{lem}[thm]{Lemma}

\theoremstyle{definition}

\theoremstyle{remark}

\newtheorem{rem}[thm]{Remark}

\newcommand{\ra}{{\rightarrow}}

\newcommand{\F}{{\mathcal F}}

\newcommand{\M}{{\mathcal M}}

\newcommand{\PP}{{\mathbb P}}

\newcommand{\X}{{\mathcal X}}

\newcommand{\Integers}{{\mathbb Z}}

\newcommand{\RationalNumbers}{{\mathbb Q}}

\newcommand{\monrep}{{mon}}

\newcommand{\LongRightArrowOf}[1]{\stackrel{#1}{\longrightarrow}}

\newcommand{\StructureSheaf}[1]{{\mathcal O}_{#1}}

\newcommand{\rank}{{\rm rank}}

\newcommand{\SheafEnd}{{\mathcal E}nd}
\newcommand{\SheafExt}{{\mathcal E}xt}

\begin{document}
\title[]
{The Standard Conjectures for holomorphic symplectic varieties
deformation equivalent to Hilbert schemes of $K3$ surfaces}
\author{Fran\c{c}ois Charles and Eyal Markman}
\address{D\'epartement de Math\'ematiques et Applications, \'Ecole Normale Sup\'erieure, 45, rue d'Ulm, 75005 Paris, France}
\email{francois.charles@ens.fr}
\address{Department of Mathematics and Statistics, 
University of Massachusetts, Amherst, MA 01003, USA}
\email{markman@math.umass.edu}

\begin{abstract}
We prove the standard conjectures for complex
projective varieties that are deformations of the Hilbert scheme of points on a
K3 surface. The proof involves Verbitsky's theory of hyperholomorphic sheaves
and a study of the cohomology algebra of Hilbert schemes of K3 surfaces.
\end{abstract}

\maketitle

\tableofcontents

%
\section{Introduction}
An {\em irreducible holomorphic symplectic manifold} is a simply connected
compact K\"{a}hler manifold $X$, such that $H^0(X,\Omega^2_X)$ is generated
by an everywhere non-degenerate holomorphic two-form
(see \cite{beauville,huybrects-basic-results}).

Let $S$ be a smooth compact K\"{a}hler $K3$ surface
and $S^{[n]}$ the Hilbert scheme (or Douady space) of length $n$ zero
dimensional subschemes of $S$. Beauville proved in \cite{beauville} that
$S^{[n]}$ is an
irreducible holomorphic symplectic manifold of dimension $2n$. If $X$ is a
smooth compact K\"{a}hler manifold deformation equivalent to
$S^{[n]}$, for some $K3$ surface $S$, then we say that $X$ is of
{\em $K3^{[n]}$-type}. The variety $X$ is then an irreducible holomorphic
symplectic manifold.
The odd Betti numbers of $X$ vanish \cite{gottsche}.

The moduli space of K\"{a}hler manifolds of
$K3^{[n]}$-type is smooth and $21$-dimensional, if $n\geq 2$, while that of $K3$
surfaces is $20$-dimensional \cite{beauville}. It follows that if $S$ is a
$K3$ surface, a general K\"ahler deformation of $S^{[n]}$ is not of the form
$S'^{[n]}$ for a K3 surface $S'$. The same goes for projective deformations.
Indeed, a general projective deformation of $S^{[n]}$ has Picard number $1$,
whereas for a projective $S$, the Picard number of $S^{[n]}$ is at least $2$.


\bigskip

In this note, we prove the standard conjectures for projective varieties of
$K3^{[n]}$-type. Let us recall general facts about the standard conjectures.

In the paper \cite{Gr69} of 1968, Grothendieck states those conjectures
concerning the existence of some algebraic cycles on smooth projective algebraic
varieties over an
algebraically closed ground field. Here we work over $\mathbb C$.
The Lefschetz standard conjecture predicts the existence of algebraic
self-correspondences on a
given smooth projective variety $X$ of dimension $d$ that give an inverse to the
operations
$$H^{i}(X) \ra H^{2d-i}(X)$$
given by the cup-product $d-i$ times with a hyperplane section,
for all $i\leq d$. Above and throughout the rest of the paper, the notation
$H^i(X)$ stands for singular cohomology with rational coefficients.

Over the complex numbers, the Lefschetz standard conjecture implies all the
standard conjectures. If it holds for a variety $X$, it implies that numerical
and homological equivalence coincide for algebraic cycles on $X$, and that the
K\"{u}nneth components of the diagonal of $X\times X$ are algebraic. We refer to
\cite{Kl68} for a detailed discussion.

\bigskip

Though the motivic picture has tremendously developed since Grothendieck's
statement of the standard conjectures, very little progress has been made in
their direction. The Lefschetz standard conjecture is known for abelian
varieties,
and in degree $1$, where it reduces to the Hodge conjecture for divisors.
The Lefschetz standard conjecture is also known for varieties $X$, for which
$H^*(X)$ is isomorphic to the Chow ring $A^*(X)$,
see \cite{Kl94}. Varieties with the latter property include flag varieties, and
smooth projective moduli spaces of sheaves on rational Poisson surfaces
\cite{ES,markman-integral-generators}.

In the paper \cite{arapura}, Arapura proves that the Lefschetz standard
conjecture holds for
uniruled threefolds, unirational fourfolds,
the moduli space of stable vector bundles over a smooth projective curve, and
for
the Hilbert scheme
$S^{[n]}$ of every smooth projective surface (\cite{arapura}, Corollaries 4.3,
7.2 and 7.5).
He also proves that if $S$ is a $K3$ or abelian surface, $H$ an ample
line-bundle on $S$,
and $\M$ a smooth and
compact moduli space of Gieseker-Maruyama-Simpson $H$-stable sheaves on $S$,
then the Lefschetz standard conjecture holds for $\M$
(\cite{arapura}, Corollary 7.9). Those results are obtained by showing that the
motive of those varieties is very close, in a certain sense, to that of a curve
or a surface. Aside from those examples and ones obtained by specific constructions from them (e.g. hyperplane
sections, products, projective bundles, etc.), very few cases of the Lefschetz standard conjecture seem to be known.

\medskip
The main result of this note is the following statement.
\begin{thm}
\label{thm-lefschetz}
The Lefschetz standard conjecture holds for every smooth projective variety
of $K3^{[n]}$-type.
\end{thm}

Since the Lefschetz standard conjecture is the strongest standard conjecture in
characteristic zero, we get the following corollary.
\begin{cor}
The standard conjectures hold for any smooth projective variety
of $K3^{[n]}$-type.
\end{cor}

Note that by the remarks above, Theorem \ref{thm-lefschetz} does not seem to follow from Arapura's
results, as a general variety of $K3^{[n]}$-type is not a moduli
space of sheaves on any K3 surface.

\bigskip

Theorem \ref{thm-lefschetz} is proven in section \ref{sec-proof-of-main-thm}.
The degree $2$ case of the  Lefschetz standard conjecture, for projective
varieties of
$K3^{[n]}$-type, has already been proven in \cite{markman-2010} as a consequence
of results of \cite{charles}. Section 2 gives general results on the Lefschetz standard conjecture. Sections 3 to 5
introduce the algebraic cycles we need for the proof, while sections 6 and 7 contain results on the cohomology
algebra of the Hilbert scheme of K3 surfaces.
%

\section{Preliminary results on the Lefschetz standard conjecture}

Let $X$ be a smooth projective
variety of dimension $d$. Let $\xi \in H^2(X)$ be the cohomology class of a
hyperplane section of $X$. According to the hard Lefschetz theorem, for all
$i\in\{0, \ldots, d\}$, cup-product with $\xi^{d-i}$ induces an isomorphism
$$L^{[d-i]}:=\cup \xi^{d-i} : H^{i}(X)\ra H^{2d-i}(X).$$
The Lefschetz standard conjecture was first stated in \cite{Gr69}, conjecture
$B(X)$. It is the following.

\begin{conj}
Let $X$ and $\xi$ be as above. Then for all $i\in\{0, \ldots, d\}$, there exists
an algebraic cycle $Z$ of codimension $i$ in the product $X\times X$ such that
the correspondence
$$[Z]_* : H^{2d-i}(X)\ra H^{i}(X)$$
is the inverse of $\cup \xi^{d-i}$.
\end{conj}
If this conjecture holds for some specific $i$ on $X$, we will say that
{\em the Lefschetz conjecture holds in degree $i$ for the variety $X$.}

We will derive Theorem \ref{thm-lefschetz} as a consequence of Theorem
\ref{thm-deformability} and Corollary \ref{cor-comparison-between-kappa-and-f} below.
In this section, we prove some general results we will need. The reader can consult \cite{arapura}, sections 1 and 4 for
related arguments, and \cite{andre} for a more general use of polarizations and 
semi-simplicity. Let us first state an easy lemma.

\begin{lem}\label{devissage}
Let $X$ be a smooth projective variety of dimension $d$. Let $i\leq d$ be an
integer.
\begin{enumerate}
\item Assume $i=2j$ is even, and let $\alpha\in H^{2j}(X)$ be the cohomology
class of a codimension $j$ algebraic cycle in $X$. Then there exists a cycle $Z$
of codimension $i=2j$ in $X\times X$ such that the image of the correspondence
$$[Z]_* : H^{2d-2j}(X)\ra H^{2j}(X)$$
contains $\alpha$.
\item Assume that $X$ satisfies the Lefschetz standard conjecture in degrees up
to $i-1$. Then $X\times X$ satisfies the Lefschetz standard conjecture in degree
up to $i-1$.

    Let $j$ and $k$ be two positive integers with $i=j+k$. Then there exists a
cycle $Z$ of codimension $i$ in $(X\times X)\times X$ such that the image of the
correspondence
$$[Z]_* : H^{4d-i}(X\times X)\ra H^{i}(X)$$
contains the image of $H^j(X)\otimes H^k(X)$ in $H^{j+k}(X)=H^i(X)$ by
cup-product.
\end{enumerate}
\end{lem}

\begin{proof}
Let $\alpha\in H^{2j}(X)$ be the cohomology class of a codimension $j$ algebraic
cycle $T$ in $X$. Let $Z$ be the codimension $i=2j$ algebraic cycle $T\times T$
in $X\times X$. Since the image in $H^i(X)\otimes H^i(X)$ of the cohomology
class of $Z$ in $H^{2i}(X\times X)$ is $\alpha\otimes \alpha$, the image of the
correspondence
$$[Z]_* : H^{2d-i}(X)\ra H^{i}(X)$$
is the line generated by $\alpha$. This proves $(1)$.

 \bigskip

Let us prove the first part of $(2)$. We repeat some of Kleiman's arguments in \cite{Kl68}.
Assume that $X$ satisfies the Lefschetz standard conjecture
in degree up to $i-1$. We want to prove that $X\times X$ satisfies the Lefschetz
standard conjecture in degree up to $i-1$. By induction, we only have to prove
that $X\times X$ satisfies the Lefschetz standard conjecture in degree
$i-1$.

For any $j$ between $0$ and $i-1$, there exists a codimension
$j$ algebraic cycle $Z_j$ in $X\times X$ such that the correspondence
$$[Z_j]_* : H^{2d-j}(X)\ra H^j(X)$$
is an isomorphism. For $k$ between $0$ and $2d$, 
let $\pi^k\in H^{2d-k}(X)\otimes H^k(X)
\subset H^{2d}(X\times X)$
be the $k$-th K\"unneth component of the diagonal. By \cite{Kl68}, Lemma 2.4,
the assumption on $X$ implies that the elements $\pi^0, \ldots, \pi^{i-1}, \pi^{2d-i+1},
\ldots, \pi^{2d}$ are algebraic. Identifying the $\pi^j$ with the correspondence they induce,
this implies that for all $j$ between $0$ and $i-1$, the projections
$$\pi^j : H^*(X)\ra H^j(X)\hookrightarrow H^*(X)$$
and
$$\pi^{2d-j} : H^*(X)\ra H^{2d-j}(X)\hookrightarrow H^*(X)$$
are given by algebraic correspondences. Replacing the correspondence $[Z_j]_*$ by
$[Z_j]_* \circ \pi_{2d-j}$, which is still algebraic, we can thus assume that the morphism
$$[Z_j]_* : H^{2d-k}(X)\ra H^{2j-k}(X)$$
induced by $[Z_j]$ is zero unless $k=j$.

Now consider the codimension $i-1$ cycle $Z$ in $(X\times X)\times(X\times X)$ defined by
$$Z=\sum_{j=0}^{i-1} Z_j \times Z_{i-1-j}.$$
We claim that the correspondence
$$[Z]_* : H^{4d-i+1}(X\times X)\ra H^{i-1}(X\times X)$$
is an isomorphism.

Fix $j$ between $0$ and $i-1$. The hypothesis on the cycles $Z_j$ imply that the correspondence
$$[Z_j\times Z_{i-1-j}]_* : H^{4d-i+1}(X\times X)\ra H^{i-1}(X\times X)$$
maps the subspace $H^{2d-k}(X)\otimes H^{2d-i+1+k}(X)$ of $H^{4d-i+1}(X\times X)$
to zero unless $k=j$, and it maps $H^{2d-j}(X)\otimes H^{2d-i+1+j}(X)$ isomorphically
onto $H^{j}(X)\otimes H^{i-1-j}(X)$. The claim follows, as does the first part of $(2)$.

\bigskip

For the second statement, let $j$ and $k$ be as in the hypothesis. Since $j$ (resp. $k$) is smaller than or equal to
$i-1$, $X$ satisfies the Lefschetz standard conjecture in degree $j$ (resp.
$k$). As a consequence, there exists a cycle $T$ (resp. $T'$) in $X\times X$
such that the morphism
$$[T]_* : H^{2d-j}(X)\ra H^{j}(X)$$
(resp. $[T']_* : H^{2d-k}(X)\ra H^{k}(X)$) is an isomorphism. Consider now the
projections $p_{13}$ and $p_{23}$ from $X\times X\times X$ to $X\times X$
forgetting the second and first factor respectively, and let $Z$ in
$CH^{i}(X\times X\times X)$ be the intersection of $p_{13}^*T$ and $p_{23}^*T'$.
 Since the cohomology class of $Z$ is just the cup-product of that of
$p_{13}^*T$ and $p_{23}^*T'$, it follows that the image of the correspondence
$$[Z]_* : H^{4d-i}(X\times X)\ra H^{i}(X)$$
contains the image of $H^j(X)\otimes H^k(X)$ in $H^{j+k}(X)=H^i(X)$ by
cup-product.
\end{proof}

The following result appears in \cite{charles}, Proposition 8. Using it with the
previous lemma, it will allow us to finish the proof of Theorem
\ref{thm-lefschetz}.

\begin{thm}\label{Lef-criterion}
Let $X$ be a smooth projective variety of dimension $d$, and let $i\leq d$ be an
integer. Then the Lefschetz conjecture is true in degree $i$ for $X$ if and only
if there exists a disjoint union $S$ of smooth projective varieties of dimension
$l\geq i$ satisfying the Lefschetz conjecture in degrees up to $i-2$ and a
codimension $i$ cycle $Z$ in $X\times S$ such that the morphism
$$[Z]_* : H^{2l-i}(S)\ra H^i(X)$$
induced by the correspondence $Z$ is surjective.
\end{thm}

\begin{cor}\label{recurrence}
Let $X$ be a smooth projective variety of dimension $d$, and let $i\leq d$ be an
integer. Suppose that $X$ satisfies the Lefschetz standard conjecture in degrees
up to $i-1$.

Let $A^i(X)\subset H^i(X)$ be the subspace of classes, which belong to
the subring generated by classes of degree $< i$, and let $Alg^i(X)\subset
H^i(X)$ be the subspace of $H^i(X)$ generated by the cohomology classes of
algebraic cycles\footnote{Note that this subspace is zero unless $i$ is even.}.

Assume that there is a cycle $Z$ of codimension $i$ in $X\times X$ such that the
image of the morphism
$$[Z]_* : H^{2d-i}(X)\ra H^{i}(X)$$
maps surjectively onto the quotient space
$H^{i}(X)/\left[Alg^i(X)+A^i(X)\right]$. Then $X$ satisfies the Lefschetz
standard conjecture in degree $i$.
\end{cor}

\begin{proof}
We use Lemma \ref{devissage}. Let $\alpha_1, \ldots, \alpha_r$ be a basis for
$Alg^i(X)$. We can find codimension $i$ cycles $Z_1, \ldots, Z_r$ in $X\times X$
and $(Z_{j,k})_{j, k>0, j+k=i}$ in $(X\times X)\times X$, such that the image of
the correspondence
$$[Z_l]_* : H^{2d-i}(X)\ra H^{i}(X)$$
contains $\alpha_l$ for $1\leq l \leq r$, and such that the image of the
correspondence
$$[Z_{j,k}]_* : H^{4d-i}(X\times X)\ra H^{i}(X)$$
contains the image of $H^j(X)\otimes H^k(X)$ in $H^{j+k}(X)=H^i(X)$, for
$j+k=i$.

\bigskip

We proved that $X\times X$ satisfies the Lefschetz standard conjecture in degree
up to $i-1$. The disjoint union of the cycles $Z\times X, (Z_l\times X)_{1\leq
l\leq r}$ and $(Z_{j,k})_{j, k>0, j+k=i}$ in a disjoint union of copies of
$(X\times X)\times X$ thus satisfies the hypothesis of Theorem
\ref{Lef-criterion} (we took products with $X$ in order to work with
equidimensional varieties). Indeed, the space generated by the images in
$H^{i}(X)$ of the correspondences $[Z_{j,k}]_*$ contains $A^{i}(X)$ by
definition. Adding the images in $H^{i}(X)$ of the $[Z_l\times X]_*$, which
generate a space containing $Alg^i(X)$, and the image in $H^{i}(X)$ of $[Z\times
X]_*$, which maps surjectively onto $H^{i}(X)/\left[A^{i}(X)+Alg^i(X)\right]$,
we get the whole space $H^{i}(X)$.

This ends the proof, and shows that $X$ satisfies the Lefschetz standard
conjecture in degree $i$.
\end{proof}

\begin{cor}
Let $X$ be a smooth projective variety with cohomology algebra generated in
degree less than $i$, and assume that $X$ satisfies the Lefschetz standard
conjecture in degree up to $i$. Then $X$ satisfies the standard conjectures.
\end{cor}

\begin{proof}
 Using induction and taking $Z=0$, the previous corollary shows that $X$
satisfies the Lefschetz standard conjecture, hence all the standard conjectures,
since we work in characteristic zero.
\end{proof}

Note that the Lefschetz conjecture is true in degree $1$, as it is a consequence
of the Lefschetz theorem on $(1,1)$-classes. The preceding corollary hence
allows us to recover the Lefschetz conjecture for abelian varieties which was
proved in \cite{lieberman}.

%
\section{Moduli spaces of sheaves on a $K3$ surface}
\label{sec-moduli-spaces-of-sheaves}
Let $S$ be a projective $K3$ surface.
Denote by $K(S,\Integers)$  the topological $K$ group of $S$, generated by
topological complex vector bundles. The $K$-group of a point is $\Integers$ and we 
let $\chi:K(S,\Integers)\rightarrow \Integers$ be the Gysin
homomorphism associated to the morphism from $S$ to a point.
The group $K(S,\Integers)$, endowed with the {\em Mukai pairing}
\[
(v,w) \ \ := \ \ -\chi(v^\vee\otimes w),
\]
is called the {\em Mukai lattice} and denoted by $\Lambda(S)$. Mukai
identifies the group $K(S,\Integers)$ with $H^*(S,\Integers)$, via the
isomorphism  sending a class $F$ to its {\em Mukai vector}
$ch(F)\sqrt{td_S}$. Using the grading of $H^*(S,\Integers)$,
the Mukai vector of $F$ is
\begin{equation}
\label{eq-Mukai-vector}
(\rank(F),c_1(F),\chi(F)-\rank(F)),
\end{equation}
where the rank is considered in
$H^0$ and $\chi(F)-\rank(F)$ in $H^4$ via multiplication by the
orientation class of $S$. The homomorphism
$ch(\bullet)\sqrt{td_S}:\Lambda(S)\rightarrow H^*(S,\Integers)$
is an isometry with respect
to the Mukai pairing on $\Lambda(S)$ and the pairing
\[
\left((r',c',s'),(r'',c'',s'')\right) \ \ = \ \
\int_{S}c'\cup c'' -r'\cup s''-s'\cup r''
\]
on $H^*(S,\Integers)$ (by the Hirzebruch-Riemann-Roch Theorem).
Mukai defines a weight $2$ Hodge structure on the Mukai lattice
$H^*(S,\Integers)$,
and hence on $\Lambda(S)$, by extending that of $H^2(S,\Integers)$,
so that the direct summands $H^0(S,\Integers)$ and $H^4(S,\Integers)$
are of type $(1,1)$ \cite{mukai-hodge}.

Let $v\in \Lambda(S)$ be a primitive class with $c_1(v)$ of Hodge-type $(1,1)$.
There is a system of hyperplanes in the ample cone of $S$, called $v$-walls,
that is countable but locally finite \cite{huybrechts-lehn-book}, Ch. 4C.
An ample class is called {\em $v$-generic}, if it does not
belong to any $v$-wall. Choose a $v$-generic ample class $H$.
Let $\M_H(v)$ be the moduli space of $H$-stable
sheaves on the $K3$ surface $S$ with class $v$.
When non-empty, the moduli space
$\M_H(v)$ is a smooth projective irreducible holomorphic symplectic variety
of $K3^{[n]}$ type, with $n=\frac{(v,v)+2}{2}$.
This result is due to several people, including
Huybrechts, Mukai, O'Grady, and Yoshioka. It can be found in its final form in
\cite{yoshioka-abelian-surface}.

Over $S\times \M_H(v)$ there exists a universal sheaf
$\F$, possibly twisted with respect to a non-trivial
Brauer class pulled-back from $\M_H(v)$.
Associated to $\F$ is a class $[\F]$ in $K(S\times \M_H(v),\Integers)$
(\cite{markman-integral-generators}, Definition 26).
Let $\pi_i$ be the projection from $S\times \M_H(v)$ onto the $i$-th factor.
Assume that $(v,v)>0$.
The second integral cohomology $H^2(\M_H(v),\Integers)$, its Hodge structure,
and its Beauville-Bogomolov pairing \cite{beauville}, are all described by
Mukai's Hodge-isometry
\begin{equation}
\label{eq-Mukai-isomorphism}
\theta \ : \ v^\perp \ \ \  \longrightarrow \ \ \  H^2(\M_H(v),\Integers),
\end{equation}
given by $\theta(x):=c_1\left(\pi_{2_!}\{\pi_1^!(x^\vee)\otimes [\F]\}\right)$
(see \cite{yoshioka-abelian-surface}). Above, $\pi_{2_!}$ and $\pi_1^!$ are
the Gysin and pull-back homomorphisms in $K$-theory.

%
\section{An algebraic cycle}
Let $\M:=\M_H(v)$ be a moduli space of stable sheaves on the $K3$ surface $S$
as in section \ref{sec-moduli-spaces-of-sheaves}, so that $\M$ is of
$K3^{[n]}$-type, $n\geq 2$.
Assume that there exists an untwisted universal sheaf
$\F$ over $S\times \M$.
Denote by $\pi_{ij}$ the projection from $\M\times S\times \M$
onto the product of the $i$-th and $j$-th factors.
Denote by $E^i$ the  relative extension sheaf
\begin{equation}
\label{eq-E}
\SheafExt^i_{\pi_{13}}\left(\pi_{12}^*\F,\pi_{23}^*\F\right).
\end{equation}
Let $\Delta\subset \M\times \M$ be the diagonal.
Then $E^1$ is a reflexive coherent
$\StructureSheaf{\M\times\M}$-module of rank $2n-2$, which is locally free
away from $\Delta$, by \cite{markman-2010}, Proposition 4.5.
The sheaf $E^0$ vanishes, while $E^2$ is isomorphic to
$\StructureSheaf{\Delta}$.
Let $[E^i]\in K(\M\times \M,\Integers)$ be the class of $E^i$, $i=1,2$, and set
\begin{equation}
\label{eq-K-class-E}
[E]:=[E^2]-[E^1].
\end{equation}
Set $\kappa(E^1):=ch(E^1)\exp\left[-c_1(E^1)/(2n-2)\right]$.
Then $\kappa(E^1)$ is independent of the choice of a universal sheaf $\F$.
Let $\kappa_i(E^1)$ be the summand in $H^{2i}(\M)$.
Then $\kappa_1(E^1)=0$.
There exists a suitable choice of $\M_H(v)$, one for each $n$, so that the sheaf
$E^1$ over $\M_H(v)\times \M_H(v)$ can be deformed, as a {\em twisted} coherent sheaf,
to a sheaf $\widetilde{E}^1$ over $X\times X$, for every $X$ of $K3^{[n]}$-type
\cite{markman-2010}. 
See \cite{caldararu-thesis} for the definition of a family of twisted sheaves.
We note here only that such a deformation is equivalent to 
a flat deformations of $\SheafEnd(E^1)$, as a reflexive coherent sheaf, together with 
a deformation of its associative algebra structure.
The characteristic class $\kappa_i(\widetilde{E}^1)$ is well defined
for twisted sheaves \cite{markman-2010}. Furthermore, $\kappa_i(\widetilde{E}^1)$
is a rational class of weight $(i,i)$, which is algebraic, whenever $X$ is
projective.
The construction is summarized in the following statement.

\begin{thm}
\label{thm-deformability}
(\cite{markman-2010}, Theorem 1.7)
Let $X$ be a smooth projective variety of $K3^{[n]}$-type.
Then there exists a smooth and proper family $\pi:\X\rightarrow C$
of irreducible holomorphic symplectic varieties, over
a simply connected reduced (possibly reducible) projective curve $C$, points
$t_0, t_1\in C$, isomorphisms $\M\cong \pi^{-1}(t_0)$ and
$X\cong \pi^{-1}(t_1)$, with the following
property. Let $p:\X\times_C\X\rightarrow C$ be the natural morphism.
The flat section of
the local system $R_{p_*}^*\RationalNumbers$ through the
class $\kappa([E^1])$ in $H^*(\M\times \M)$ is algebraic in $H^*(X_t\times
X_t)$,
for every projective fiber $X_t$, $t\in C$, of $\pi$.
\end{thm}

Verbitsky's theory of hyperholomorphic sheaves plays a central role in
the proof of the above theorem, see \cite{kaledin-verbitski-book}.

%
\section{A self-adjoint algebraic correspondence}
Let $K(\M)$ be the topological $K$-group with $\RationalNumbers$ coefficients.
Define the {\em Mukai pairing} on $K(\M)$ by
$(x,y):=-\chi(x^\vee\otimes y)$.
Let $D_\M:H^*(\M)\rightarrow H^*(\M)$ be the automorphism
acting on $H^{2i}(\M)$ via multiplication by $(-1)^i$.
Define the {\em Mukai pairing} on $H^*(\M)$ by
\[
(\alpha,\beta) \ \ \ := \ \ \ -\int_{\M}D_\M(\alpha)\beta.
\]
Define
\begin{equation}
\label{eq-mu}
\mu \ : \ K(\M) \ \ \ \longrightarrow \ \ \ H^*(\M)
\end{equation}
by $\mu(x):=ch(x)\sqrt{td_\M}$.
Then $\mu$ is an isometry, by the Hirzebruch-Riemann-Roch theorem.

\begin{rem}
Note that the graded direct summands $H^i(\M)$ of the cohomology ring
$H^*(\M)$ satisfy the usual orthogonality relation with respect to the Mukai
pairing:
$H^i(\M)$ is orthogonal to $H^j(\M)$, if $i+j\neq 4n$.

This is the main reason for the above definition of the Mukai pairing.
The Chern character $ch:K(\M)\rightarrow H^*(\M)$ is an isometry
with respect to another pairing
$(x,y):=-\int_{\M}D_\M(x)\cdot y \cdot td_\M$.
However, the graded summands $H^i(\M)$ need not satisfy
the orthogonality relations above.
\end{rem}

Given two $\RationalNumbers$-vector spaces $V_1$ and $V_2$,
each endowed with a non-degenerate symmetric bilinear pairing
$(\bullet,\bullet)_{V_i}$, and a homomorphism $h:V_1\rightarrow V_2$, we denote
by
$h^\dagger$ the adjoint operator, defined by the equation
\[
(x,h^\dagger(y))_{V_1}  \ \ := \ \ (h(x),y)_{V_2},
\]
for all $x\in V_1$ and $y\in V_2$.
Set $K(S):=K(S,\Integers)\otimes_\Integers\RationalNumbers$.
We consider  $H^*(\M)$, $K(\M)$, and
$K(S)$,
all as vector spaces over $\RationalNumbers$ endowed with the Mukai pairing.

Let $\pi_i$ be the projection from $\M\times \M$ onto the $i$-th factor,
$i=1,2$.
Define
\[
\tilde{f}' \ : \ K(\M) \ \ \ \rightarrow \ \ \ K(\M),
\]
by $\tilde{f}'(x):=\pi_{2_!}(\pi_1^!(x)\otimes [E])$, where $[E]$ is the class
given in (\ref{eq-K-class-E}), and $\pi_{2_!}$ and $\pi_1^!$ are 
the Gysin and pull-back homomorphisms in $K$-theory.
Let $p_i$ be the projection from $S\times\M$ onto the $i$-th factor.
Define
\[
\phi' \ : \ K(S) \ \ \ \rightarrow \ \ \ K(\M)
\]
by $\phi'(\lambda):=p_{2_!}(p_1^!(\lambda)\otimes [\F])$.
Define
\[
\psi' \ : \ K(\M) \ \ \ \rightarrow \ \ \ K(S)
\]
by $\psi'(x):=p_{1_!}(p_2^!(x)\otimes [\F^\vee])$,
where $\F^\vee$ is the dual class.
We then have the following identities
\begin{eqnarray}
\label{eq-psi-is-adjoint-of-phi}
\psi' & = & (\phi')^\dagger,
\\
\label{eq-self-duality}
\tilde{f}' & = & \phi'\circ\psi'.
\end{eqnarray}
Equality (\ref{eq-psi-is-adjoint-of-phi})  is a $K$-theoretic analogue of
the following well known
fact in algebraic geometry. Let $\Phi:D^b(S)\rightarrow D^b(\M)$
be the Fourier-Mukai functor with kernel $\F$.
Set $\F_R:=\F^\vee\otimes p_1^*\omega_S[2]$ and let
$\Psi:D^b(\M)\rightarrow D^b(S)$ be the Fourier-Mukai functor with kernel
$\F_R$.
Then $\Psi$ is the right adjoint functor of $\Phi$
(\cite{mukai-duality} or \cite{huybrechts-book-FM}, Proposition 5.9).
The classes of $\F^\vee$ and $\F_R$ in $K(S\times \M)$ are equal, since
$\omega_S$ is trivial.
The equality (\ref{eq-psi-is-adjoint-of-phi}) is proven using the same argument
as its
derived-category analogue.
Equality (\ref{eq-self-duality}) expresses the fact that the class $[E]$ is the
convolution
of the classes of $\F^\vee$ and $\F$.
We conclude that $\tilde{f}'$ is {\em self adjoint.}
Set
\[
f' \ \ := \ \ \mu \circ \tilde{f}' \circ \mu^\dagger.
\]
Then $f'$ is the self adjoint endomorphism  given by the algebraic class
\[
\left(\pi_1^*\sqrt{td_\M}\right)ch([E])\left(\pi_2^*\sqrt{td_\M}\right)
\]
in $H^*(\M\times\M)$.

We normalize next the endomorphism $f'$ to an endomorphism $f$. The
latter will be shown to have a monodromy-equivariance property in section 
\ref{sec-monodromy}.
Let $\alpha\in K(\M)$ be the class  satisfying
$ch(\alpha)=\exp\left(\frac{-c_1(\phi'(v^\vee))}{2n-2}\right)$.
Note that $\alpha$ is the class of a $\RationalNumbers$-line-bundle.
Let $\tau_\alpha:K(\M)\rightarrow K(\M)$ be tensorization with $\alpha$, i.e.,
$\tau_\alpha(x):=x\otimes\alpha$. Then $\tau_\alpha$ is an isometry. Hence,
$\tau_\alpha^\dagger=\tau_\alpha^{-1}$.
Set
\begin{eqnarray*}
\phi & := & \tau_\alpha\circ \phi',
\\
\psi & := & \psi'\circ \tau_\alpha^{-1},
\\
\tilde{f} & := & \phi\circ \psi,
\\
f & := & \mu \circ \tilde{f}\circ \mu^\dagger.
\end{eqnarray*}
Then $f$ is the self adjoint endomorphism given by the algebraic class
\begin{equation}
\label{eq-normalized-class-E-and-two-square-roots}
\pi_1^*\exp\left(\frac{c_1(\phi'(v^\vee))}{2n-2}\right)
\left(\pi_1^*\sqrt{td_\M}\right)ch([E])\left(\pi_2^*\sqrt{td_\M}\right)
\pi_2^*\exp\left(\frac{-c_1(\phi'(v^\vee))}{2n-2}\right).
\end{equation}

Let $h_i:H^*(\M)\rightarrow H^{2i}(\M)$ be the projection, and
$e_i:H^{2i}(\M)\rightarrow H^*(\M)$ the inclusion.
Set
\[
f_i \ \ \  := \ \ \  h_i\circ f\circ e_{2n-i}.
\]
Note that $f_i$ is induced by the graded summand
in $H^{4i}(\M\times\M)$ of
the class given in equation (\ref{eq-normalized-class-E-and-two-square-roots}).

%
\section{Generators for the cohomology ring and the image of $f_i$}
Let $A^{2i}\subset H^{2i}(\M)$ be the subspace of classes, which belong to
the subring generated by classes of degree $< 2i$. Set
\[
\overline{H}^{2i}(\M)  \ \ \ := \ \ \
H^{2i}(\M)/\left[A^{2i}+\RationalNumbers\cdot c_i(TX)\right].
\]

\begin{prop}
\label{prop-surjectivity}
The composition
\begin{equation}
\label{eq-bar-f-i}
H^{4n-2i}(\M) \ \LongRightArrowOf{f_i} \ H^{2i}(\M) \ \rightarrow \
\overline{H}^{2i}(\M)
\end{equation}
is surjective, for $i\geq 2$.
\end{prop}

The proposition is proven below after Claim \ref{claim-image-of-f-i}.
Let $g_i:H^{4n-2i}(\M)\rightarrow H^{2i}(\M)$ be the homomorphism induced by
the graded summand of degree $4i$ of the cycle
\begin{equation}
\label{eq-deformable-class}
-\left(\pi_1^*\sqrt{td_\M}\right)\kappa(E^1)\left(\pi_2^*\sqrt{td_\M}\right).
\end{equation}
Denote by $\bar{f}_i:H^{4n-2i}(\M)\rightarrow\overline{H}^{2i}(\M)$
the homomorphism given in (\ref{eq-bar-f-i}) and define
$\bar{g}_i:H^{4n-2i}(\M)\rightarrow\overline{H}^{2i}(\M)$ similarly in terms of
$g_i$.

\begin{cor}
\label{cor-comparison-between-kappa-and-f}
$\bar{g}_i=\bar{f}_i$, for $i\geq 2$. In particular, $\bar{g}_i$ is surjective,
for $i\geq 2$.
\end{cor}

\begin{proof}
The equality
$
c_1([E])=\pi_1^*c_1(\phi'(v^\vee))-\pi_2^*c_1(\phi'(v^\vee))
$
is proven in \cite{markman-2010}, Lemma 4.3. Hence,
the difference between the two classes
(\ref{eq-normalized-class-E-and-two-square-roots}) and
(\ref{eq-deformable-class}) is $ch(\StructureSheaf{\Delta})\pi_1^*td_{\M}$.
Now $ch_j(\StructureSheaf{\Delta})=0$, for $0\leq j < 2n$.
Hence, $f_i=g_i$, for $0\leq i \leq n-1$.
The quotient group $\overline{H}^{2i}(\M)$ vanishes, for $i>n-1$,
by \cite{markman-generators}, Lemma 10, part 4.
Consequently, $\bar{f}_i=0=\bar{g}_i$, for $i\geq n$.
%
\end{proof}

Set $\eta:=\mu\circ \phi$, where $\mu$ is given in equation (\ref{eq-mu}).
Then $\eta^\dagger=\psi\circ \mu^\dagger$, and we have
\[
f \ \ = \ \ \eta\circ \eta^\dagger.
\]
Set $\eta_i:=h_i\circ \eta$.
We abuse notation and identify $h_i$ with the endomorphism $e_i\circ h_i$
of $H^*(\M)$. Similarly, we identify $e_{2n-i}$ with  the endomorphism
$e_{2n-i}\circ h_{2n-i}$ of $H^*(\M)$. With this notation we have
\[
h_i^\dagger=e_{2n-i}.
\]

We get the following commutative diagram.
\begin{equation}
\label{eq-diagram}
\xymatrix{
H^{4n-2i}(\M) \ar[rr]^{f_i} \ar[d]^{e_{2n-i}}
\ar@/_7pc/[dddr] _{\eta_i^\dagger} & &
H^{2i}(\M)
\\
H^*(\M) \ar[d]^{\mu^\dagger} \ar[rr]^{f}
& &
H^*(\M) \ar[u]^{h_i}
\\
K(\M) \ar[rr]^{\tilde{f}} \ar[dr]_{\psi} & &
K(\M) \ar[u]^{\mu}_{\cong}
\\
& K(S) \ar[ur]_{\phi}
\ar@/_7pc/[uuur] _{\eta_i}
 }
\end{equation}

The two main ingredients in the proof of
Proposition \ref{prop-surjectivity} are the following Theorem and the monodromy
equivariance of Diagram (\ref{eq-diagram}) reviewed in section
\ref{sec-monodromy}.

\begin{thm}
\label{thm-generators}
The composite homomorphism
\[
K(S)\ \LongRightArrowOf{\eta_i} \ H^{2i}(\M) \
\rightarrow \ \overline{H}^{2i}(\M)
\]
is surjective, for all $i\geq 1$.
\end{thm}

\begin{proof}
The subspaces $ch_i(\phi'(K(S))$, $i\geq 1$, generate the
cohomology ring $H^*(\M)$, by (\cite{markman-generators}, Corollary 2).
When $\M=S^{[n]}$, this was proven independently in \cite{lqw}.
The same statement holds for the subspaces $ch_i(\phi(K(S))$.
Indeed, $ch_1(\phi(K(S))=ch_1(\phi'(K(S))=
H^2(\M)$, since $\phi'(\lambda^\vee)$ is a class of rank $0$, for $\lambda\in
v^\perp$,
and so $c_1(\phi'(\lambda^\vee))=c_1(\phi(\lambda^\vee))$, for
$\lambda\in v^\perp$. Now $ch_1(\phi'([v^\perp]^\vee))=H^2(\M)$,
since Mukai's isometry given in
(\ref{eq-Mukai-isomorphism}) is surjective.
For $i>1$, the subspaces
$ch_i(\phi'(K(S))$ and $ch_i(\phi(K(S))$
are equal modulo the subring generated by
$H^2(\M)$. The surjectivity of the composite homomorphism follows.
\end{proof}

\begin{claim}
If $\eta_i$ is injective, then $Im(f_i)=Im(\eta_i)$.
\end{claim}

\begin{proof}
The assumption implies that $\eta_i^\dagger$ is surjective. Furthermore, we have
$
f_i=
\eta_i\circ\eta_i^\dagger.
$
The equality $Im(f_i)=Im(\eta_i)$ follows.
\end{proof}

In the next section we will prove an analogue of the above claim,
without the assumption that $\eta_i$ is injective (see Claim
\ref{claim-image-of-f-i}).

%
\section{Monodromy}
\label{sec-monodromy}
Recall that the Mukai lattice $\Lambda(S)$ is a rank $24$ integral lattice isometric
to the orthogonal direct sum $E_8(-1)^{\oplus 2}\oplus U^{\oplus 4}$,
where $E_8(-1)$ is the negative definite $E_8$ lattice and $U$ is the
unimodular rank $2$ hyperbolic lattice \cite{mukai-hodge}.
Recall that $\M$ is the moduli space $\M_H(v)$.
Denote by $O^+\Lambda(S)_v$ the subgroup of isometries of the Mukai lattice,
which send $v$ to itself and preserve the spinor norm.
The {\em spinor norm} is the character
$O\Lambda(S)\rightarrow \{\pm 1\}$, which sends reflections by $-2$ vectors to
$1$ and
reflections by $+2$ vectors to $-1$.
The group $O^+\Lambda(S)_v$ acts on $\Lambda(S)$ and on
$K(S)\cong \Lambda(S)\otimes_\Integers\RationalNumbers$ via the natural action.

\hide{
Let $\theta:v^\perp\rightarrow H^2(\M,\Integers)$ be the Mukai isomorphism,
given in (\ref{eq-Mukai-isomorphism}).
Note the $c_1(\phi(\lambda^\vee))=\theta(\lambda)$, for all $\lambda\in
v^\perp$.
Let $w\in v^\perp$ be the class satisfying $c_1(\phi(v^\vee))=\theta(w)$.
Given an isometry $g\in O^+\Lambda(S)_v$, let $\ell_g$ be the complex
topological
line bundle on $\M$ satisfying
\[
c_1(\ell_g) \ \ \ := \ \ \ g(w)-
\]
}

Let $D_S:K(S)\rightarrow K(S)$ be given by
$D_S(\lambda)=\lambda^\vee$.

\begin{thm}
\label{thm-symmetries}
\begin{enumerate}
\item
(\cite{markman-monodromy-I}, Theorem 1.6)
There exist  natural homomorphisms
\begin{eqnarray*}
\monrep \ : \ O^+\Lambda(S)_v  & \longrightarrow & GL\left[H^*(\M)\right],
\\
\widetilde{\monrep} \ : \ O^+\Lambda(S)_v  & \longrightarrow &
GL\left[K(\M)\right],
\end{eqnarray*}
introducing an action of $O^+\Lambda(S)_v$ on both $H^*(\M)$  and $K(\M)$
via monodromy operators. Denote the image of $g\in O^+\Lambda(S)_v$
by $\monrep_g$ and $\widetilde{\monrep}_g$.
\item
\label{thm-item-automorphic-factor}
(\cite{markman-monodromy-I}, Theorem 3.10) The equation
\begin{equation}
\label{eq-equivariance}
\widetilde{\monrep}_g(\phi(\lambda^\vee))=\phi\left([g(\lambda)]^\vee\right)
\end{equation}
holds for every $g\in O^+\Lambda(S)_v$, for all $\lambda\in \Lambda(S)$.
Consequently, the composite homomorphism
\[
K(S) \ \LongRightArrowOf{D_S}
K(S) \ \LongRightArrowOf{\eta} H^*(M)
\]
is $O^+\Lambda(S)_v$ equivariant.\footnote{The appearance of $\lambda^\vee$
instead
of $\lambda$ as an argument of $\phi$ in (\ref{eq-equivariance}),
as well as in equation (\ref{eq-Mukai-isomorphism}) for Mukai's isometry,
is due to the fact that we use the Mukai pairing
to identify $\Lambda(S)$ with its dual.
So the class of the kernel 
$[\F]\otimes p_2^*\exp\left(\frac{-c_1(\phi'(v^\vee))}{2n-2}\right)$ of $\phi$ in
$K(S\times \M)\cong K(S)\otimes K(\M)$ is $O^+\Lambda(S)_v$ invariant
with respect to the usual action of $O^+\Lambda(S)_v$ on the first factor, and
the monodromy action on the second.}
\end{enumerate}
\end{thm}


Set $w:=D_S(v)$.
We have the orthogonal direct sum decomposition
\[
K(S) \ \ = \ \
\RationalNumbers w \oplus w^\perp_{\RationalNumbers}
\]
into two distinct irreducible representations of $O^+\Lambda(S)_v$, where we
consider
a new action of $O^+\Lambda(S)_v$ on $K(S)$, i.e., the conjugate by $D_S$ of the
old one.
So $g\in O^+\Lambda(S)_v$  acts on $K(S)$ via $D_S\circ g\circ D_S$.
Let $\pi_w : K(S)\rightarrow \RationalNumbers w$ and
$\pi_{w^\perp} : K(S)\rightarrow w^\perp_{\RationalNumbers}$
be the orthogonal projections.
Let $\eta_w:\RationalNumbers w\rightarrow H^*(\M)$
be the restriction of $\eta$ to $\RationalNumbers w$
and $\eta_{w^\perp}:w^\perp_\RationalNumbers\rightarrow H^*(\M)$
the restriction of $\eta$ to $w^\perp_{\RationalNumbers}$.
We have
\[
\pi_w\circ \eta^\dagger \ = \ (\eta_w)^\dagger \ \ \ \mbox{and} \ \ \
\pi_{w^\perp}\circ \eta^\dagger \ = \ (\eta_{w^\perp})^\dagger.
\]
Set $\eta_{i,w}:=h_i\circ\eta_w$ and $\eta_{i,w^\perp}:=h_i\circ
\eta_{w^\perp}$.
Then
\begin{eqnarray*}
(\eta_{i,w})^\dagger & = & \pi_w\circ \eta^\dagger\circ e_{2n-i},
\\
(\eta_{i,w^\perp})^\dagger & = & \pi_{w^\perp}\circ \eta^\dagger\circ e_{2n-i}.
\end{eqnarray*}

\begin{claim}
\label{claim-image-of-f-i}
The homomorphisms $f_i$ and $\eta_i$ in diagram
(\ref{eq-diagram}) have the same image in $H^{2i}(\M)$.
\end{claim}

\begin{proof}
Clearly, $\eta_{i,w}$ is injective, if it does not vanish.
We observe next that the same is true for $\eta_{i,w^\perp}$.
This follows from the fact that $\eta_{i,w^\perp}$
is equivariant with respect to the action of
the group $O^+\Lambda(S)_v$
(Theorem \ref{thm-symmetries}, part \ref{thm-item-automorphic-factor}).
Now $w^\perp_{\RationalNumbers}$ is an irreducible representation of
$O^+\Lambda(S)_v$.
Hence, $\eta_{i,w^\perp}$ is injective, if and only if it does not
vanish. We have
\begin{eqnarray*}
\eta_i&=&\eta_{i,w}\circ\pi_w+\eta_{i,w^\perp}\circ\pi_{w^\perp},
\\
(\eta_i)^\dagger&=&(\eta_{i,w})^\dagger+(\eta_{i,w^\perp})^\dagger, \ \mbox{and}
\\
f_i &=& \eta_{i,w}\circ (\eta_{i,w})^\dagger+
\eta_{i,w^\perp}\circ (\eta_{i,w^\perp})^\dagger.
\end{eqnarray*}
Furthermore, the image of $(\eta_{i,w})^\dagger$ is equal to $\RationalNumbers
w$,
if $\eta_{w,i}$ does not vanish,
and the image of $(\eta_{i,w^\perp})^\dagger$ is equal to
$w^\perp_\RationalNumbers$,
if $\eta_{i,w^\perp}$ does not vanish.
Hence, the image of $\eta_{i,w}\circ (\eta_{i,w})^\dagger$ is equal to the image
of
$\eta_{i,w}$ and the image of
$\eta_{i,w^\perp}\circ (\eta_{i,w^\perp})^\dagger$ is equal to the image of
$\eta_{i,w^\perp}$.
The image of $f_i$ is thus equal to
the sum of the images of $\eta_{i,w}$ and $\eta_{i,w^\perp}$.
The latter is precisely the image of $\eta_i$.
\end{proof}

\begin{proof}
(Of Proposition \ref{prop-surjectivity})
Follows immediately from
Theorem \ref{thm-generators} and Claim \ref{claim-image-of-f-i}.
\end{proof}

%

\section{Proof of the main theorem}
\label{sec-proof-of-main-thm}

We can now prove the main result of this note. We use the notations of section 2.

\begin{proof}[Proof of Theorem \ref{thm-lefschetz}]
Let $X$ be a smooth projective variety of $K3^{[n]}$-type. According to Theorem
\ref{thm-deformability}, there exists a smooth and proper family
$p:\X\rightarrow C$ of irreducible holomorphic symplectic varieties, over a
connected reduced projective curve $C$, points $t_1, t_2\in C$, isomorphisms
$\M\cong p^{-1}(t_1)$ and $X\cong p^{-1}(t_2)$. Additionally, if
$q:\X\times_C\X\rightarrow C$ is the natural morphism, there exists a flat
section $s$ of the local system $R_{q_*}^*\RationalNumbers$ through the class
$\kappa([E^1])$ in $H^*(\M\times \M)$ which is algebraic in $H^*(X_t\times
X_t)$, for every projective fiber $X_t$, $t\in C$, of $p$.

Let us denote by $Z_i$ an algebraic cycle in $H^{2i}(X\times X)$ with cohomology
class the degree $2i$ component of $\left(\pi_1^*\sqrt{td_X}\right)
s(t_2)\left(\pi_2^*\sqrt{td_X}\right)$, where $\pi_1$ and $\pi_2$ are the two
projections $X\times X \ra X$.

\bigskip

Using the cycles $Z_i$, we prove by induction on $i\leq n$ that $X$ satisfies
the Lefschetz standard conjecture in degree $2i$ for every integer $i$ -- recall
that the cohomology groups of $X$ vanish in odd degrees. This is obvious for
$i=0$.

Let $i\leq n$ be a positive integer. Assume that the Lefschetz conjecture holds
for $X$ in degrees up to $2i-1$. Let us show that the morphism
$$H^{4n-2i}(X) \ \LongRightArrowOf{[Z_i]_*} \ H^{2i}(X) \ \rightarrow
H^{2i}(X)/\left[A^{2i}(X)+\RationalNumbers\cdot c_i(TX)\right]$$
is surjective. Since $p$ and $q$ are smooth, the morphism above is the fiber at
$t_2$ of the morphism
$$R^{4n-2i}p_*\RationalNumbers \ra R^{2i}p_*\RationalNumbers$$
of local systems over $C$ induced by $s$, which implies that it is surjective at
$t_2$ if and only if it is surjective at $t_1$. The fiber at $t_1$ of this
morphism is induced by the class $\kappa([E^1])$, which shows that it is
surjective by Corollary \ref{cor-comparison-between-kappa-and-f}.

Corollary \ref{recurrence} now shows that $X$ satisfies the Lefschetz standard
conjecture in degree $2i$, which concludes the proof.
\end{proof}

\begin{rem}
Note that the proof of the main result of this note makes essential use of 
deformations of hyperk\"ahler varieties along twistor lines,
and that a general deformation of a hyperk\"ahler variety along a twistor line is 
never algebraic, see \cite{huybrects-basic-results}, 1.17.
Though the standard conjectures deal with projective varieties, we do not know 
a purely algebraic proof of the result of this note.
\end{rem}

{\bf Acknowledgements:} The work on this paper began during
the authors visit at the University of Bonn in May, 2010.
We would like to thank Daniel Huybrechts
for the invitations, for his hospitality, and for his contribution to the
stimulating conversations the three of us held regarding the paper
\cite{charles}. The first author also wants to thank Claire Voisin for 
numerous discussions.


\end{document}